\documentclass[11pt]{amsart}

\usepackage[T1]{fontenc}
\usepackage[utf8]{inputenc}
\usepackage{geometry}
\usepackage{amssymb}
\usepackage{amsmath}
\usepackage{mathtools}
\usepackage{xcolor}
\usepackage[colorlinks=true,citecolor=red,linkcolor=blue,urlcolor=red]{hyperref}

\DeclareMathOperator{\tr}{tr}
\DeclareMathOperator{\lct}{lct}
\DeclareMathOperator{\glct}{glct}
\DeclareMathOperator{\End}{End}
\DeclareMathOperator{\Aut}{Aut}
\DeclareMathOperator{\Gr}{Gr}

\newcommand{\HH}{\mathcal{H}}

\newtheorem{thm}{Theorem}[section]
\newtheorem{lem}[thm]{Lemma}
\newtheorem{prop}[thm]{Proposition}

\theoremstyle{definition}

\newtheorem{ques}[thm]{Question}
\newtheorem{rem}[thm]{Remark}

\numberwithin{equation}{section}

\begin{document}

\title[An equivariant fixed-level Demailly identity]{An equivariant fixed-level Demailly identity for Fano manifolds}

\author{Jihao Liu}
\address{Department of Mathematics, Peking University, No. 5 Yiheyuan Road, Haidian District, Beijing 100871, China}
\address{Beijing International Center for Mathematical Research, Peking University, No. 5 Yiheyuan Road, Haidian District, Beijing 100871, China}
\email{liujihao@math.pku.edu.cn}

\author{Sheng Qin}
\address{Department of Mathematics, Peking University, No. 5 Yiheyuan Road, Haidian District, Beijing 100871, China}
\email{qinsheng@stu.pku.edu.cn}

\subjclass[2020]{32Q20, 14J45, 32U05, 14L24}
\keywords{Tian's alpha invariant, global log canonical threshold, Bergman potential, Fano manifold, complex singularity exponent, equivariant}
\date{\today}

\begin{abstract}
For a Fano manifold $X$, a compact group $G\subset\Aut(X)$, and a fixed level $k$ with $-kK_X$ basepoint-free, Jin and Rubinstein asked whether the fixed-level equivariant Tian's alpha invariant $\alpha_{k,G}$ equals the fixed-level equivariant global log canonical threshold $\glct_{k,G}$, and proved this equality for toric $X$. In this paper we provide a positive answer to Jin and Rubinstein's question in full generality. The main result of this paper was obtained by Chatgpt 5.5 pro, and the Danus system based on the Rethlas system.
\end{abstract}

\maketitle

\section{Introduction}\label{sec:introduction}

We work over the field of complex numbers $\mathbb C$.

Let $X$ be a Fano manifold of dimension $n$, let $L=-K_X$ be its anticanonical line bundle, and let $G\subset\Aut(X)$ be a compact subgroup. Fix a positive integer $k$ such that the level-$k$ anticanonical bundle $-kK_X$ is basepoint-free, and write $E=H^0(X,-kK_X)$. Choosing a $G$-invariant Hermitian metric $h$ on $-kK_X$, a $G$-invariant Hermitian inner product on $E$, and a smooth positive volume measure $\mu$ on $X$, Jin and Rubinstein \cite[Definitions~4.1 and~4.4]{JR25} attach to this data two thresholds.

The analytic threshold is the \emph{fixed-level equivariant Tian's alpha invariant}. Let
\begin{equation}\label{eq:Hspace}
\HH_{k,G}=\left\{\varphi=\frac1k\log\sum_{i=1}^{N}|s_i|_h^2:\ \{s_i\}\text{ is a basis of }E,\ \varphi\text{ is }G\text{-invariant}\right\},
\end{equation}
where $N=\dim E$, and set
\begin{equation}\label{eq:alpha}
\alpha_{k,G}=\sup\left\{c>0:\ \sup_{\varphi\in\HH_{k,G}}\int_X e^{-c(\varphi-\sup_X\varphi)}\,d\mu<\infty\right\}.
\end{equation}
The algebraic threshold is the \emph{fixed-level equivariant global log canonical threshold}. For a nonzero $G$-invariant linear subspace $W\subset E$ with basis $w_1,\dots,w_p$, set
\begin{equation}\label{eq:lctW}
\lct|W|=\sup\left\{a>0:\ \Bigl(\textstyle\sum_{j=1}^{p}|w_j|_h^2\Bigr)^{-a}\text{ is locally integrable on }X\right\},
\end{equation}
which is independent of the chosen basis, and put
\begin{equation}\label{eq:glct}
\glct_{k,G}=k\inf_{\substack{0\neq W\subset E\\gW=W\ \forall g\in G}}\lct|W|.
\end{equation}
These definitions appear, in this normalization, as \cite[Definition~4.1, Lemma~4.3, and Definition~4.4]{JR25}. The fixed-level quantities differ from the limit identity $\alpha_G=\inf_k\glct_{k,G}$.

The relationship between Tian's alpha invariants and algebraic global log
canonical thresholds is by now classical; see
\cite{Tia87,Tia90,DK01,CS08,Shi10,JR25} and the references therein for related results and background.  The fixed-level equivariant case, however, was left open by the existing comparison results: Jin and Rubinstein introduced $\glct_{k,G}$ as the algebraic counterpart of $\alpha_{k,G}$ and
asked whether $\glct_{k,G}=\alpha_{k,G}$ \cite[Problem~1.2]{JR25}.

\begin{ques}[{\cite[Problem~1.2]{JR25}}]\label{ques:jr}
Let $X$ be Fano with $L=-K_X$, and let $G\subset\Aut(X)$ be a compact subgroup. Is $\glct_{k,G}=\alpha_{k,G}$?
\end{ques}

The non-equivariant comparison of Tian's alpha invariant with the global log canonical threshold, and the equivariant comparison in the limit $k\to\infty$, are classical; the fixed-level equivariant comparison at a single level $k$ was open, and is raised as Problem~1.2 in \cite[Problem~1.2]{JR25}. Jin and Rubinstein establish Question~\ref{ques:jr} for toric Fano manifolds under suitable compact group actions, where the toric structure makes the invariant section spaces explicit, and remark that the general non-toric case requires a corresponding structure result for the spaces of invariant sections. The purpose of this note is to prove the equality for an arbitrary compact complex manifold and an arbitrary basepoint-free line bundle with a compact group action.

We now record the same definitions for a general basepoint-free line bundle.  Let $X$ be a compact complex manifold, let $L$ be a basepoint-free holomorphic line bundle on $X$, and let $G$ be a compact group acting holomorphically on $X$ and linearly on $L$.  Fix a smooth $G$-invariant Hermitian metric $h$ on $L$, a smooth positive volume measure $\mu$ on $X$, and set $E=H^0(X,L)$, $N=\dim E$.  Define
\[
\HH_{k,G}(X,L)
=
\left\{
\varphi=\frac1k\log\sum_{i=1}^{N}|s_i|_h^2:
\{s_i\}\text{ is a basis of }H^0(X,L),\ 
\varphi\text{ is }G\text{-invariant}
\right\}.
\]
Then put
\[
\alpha_{k,G}(X,L)
=
\sup\left\{
c>0:
\sup_{\varphi\in\HH_{k,G}(X,L)}
\int_X e^{-c(\varphi-\sup_X\varphi)}\,d\mu<\infty
\right\}.
\]
For a nonzero $G$-invariant linear subspace $W\subset H^0(X,L)$ with basis
$w_1,\dots,w_p$, set
\[
\lct|W|
=
\sup\left\{
a>0:
\left(\sum_{j=1}^{p}|w_j|_h^2\right)^{-a}
\text{ is locally integrable on }X
\right\},
\]
and define
\[
\glct_{k,G}(X,L)
=
k\inf_{\substack{0\neq W\subset H^0(X,L)\\ gW=W\ \forall g\in G}}
\lct|W|.
\]
When the line bundle \(L\) is clear from the context, we suppress it from the notation and simply write \(\HH_{k,G}\), \(\alpha_{k,G}\), and \(\glct_{k,G}\).

When $X$ is Fano and $L=-kK_X$, these definitions recover the fixed-level
anticanonical invariants $\alpha_{k,G}$ and $\glct_{k,G}$ above.

\begin{thm}\label{thm:general}
Let $X$ be a compact complex manifold with a smooth positive volume measure $\mu$, let $L$ be a basepoint-free holomorphic line bundle on $X$ with a smooth $G$-invariant Hermitian metric $h$, where $G$ is a compact group acting holomorphically on $X$ and linearly on $L$, and let $k$ be a positive integer. With $\alpha_{k,G}$ and $\glct_{k,G}$ defined as above, one has $\alpha_{k,G}=\glct_{k,G}$.
\end{thm}


The Fano case gives the following corollary.

\begin{thm}\label{thm:main}
Let $X$ be a Fano manifold and $G\subset\Aut(X)$ a compact subgroup. Fix a positive integer $k$ such that $-kK_X$ is basepoint-free. Then, with the thresholds defined in \eqref{eq:alpha} and \eqref{eq:glct},
\begin{equation}\label{eq:equality}
\alpha_{k,G}=\glct_{k,G}.
\end{equation}
In particular, Question~\ref{ques:jr} has an affirmative answer for every Fano manifold $X$ and every compact subgroup $G\subset\Aut(X)$.
\end{thm}

The proof uses only the $G$-linearization of $-kK_X$ and the basepoint-freeness of the linear system. It makes no use of a toric structure, of monomial bases, of finite group orbits, or of the classification of Fano manifolds.

The structure of the argument is as follows. We first parametrize the invariant Bergman potentials in \eqref{eq:Hspace} by positive Hermitian frame operators, and observe that, after Haar averaging, an invariant Bergman potential is represented by a frame operator commuting with the unitary representation on $E$ (Section~\ref{sec:averaging}). The upper bound $\alpha_{k,G}\le\glct_{k,G}$ is then a degeneration: approaching the orthogonal projection onto a worst invariant subspace through positive definite operators produces invariant potentials whose integrability collapses exactly like the subspace linear system (Section~\ref{sec:upper}). The lower bound $\alpha_{k,G}\ge\glct_{k,G}$ is the crux. A trace-normalized commuting frame operator has a top eigenspace which is automatically $G$-invariant and carries at least a $1/N$ fraction of the pointwise squared norm; the desired uniform integrability over \emph{all} invariant Bergman potentials is then reduced to a single uniform integrability bound over the \emph{compact} family of invariant subspaces, which we obtain from the effective semicontinuity of complex singularity exponents of Demailly and Koll\'ar \cite[Main Theorem~0.2(2)]{DK01} (Sections~\ref{sec:uniform} and~\ref{sec:lower}).

\begin{rem}\label{rem:rethlas}
The sketch of the proof of the main result of this paper was obtained by Chatgpt 5.5 pro, and later summed up, verified, and properly written by the Danus system, a specialized agent built on Rethlas and substantially more capable for fundamental mathematical research
based on the Rethlas system. Human verification and polishing were done afterwards. See \cite{Ju+26} for a detailed introduction to the Rethlas system. Due to the limitation of automated systems, it is possible that we have missed some related references in the literature, and we welcome any comments from experts.
\end{rem}

\subsection*{Acknowledgements}
The authors were partially supported by the National Key R\&D Program of China \#\allowbreak 2024YFA1014400.
The authors would like to thank the Rethlas team, namely Haocheng Ju, Jiedong Jiang, Shurui Liu, Guoxiong Gao, Yuefeng Wang, Zeming Sun, Bin Wu, Liang Xiao, and Bin Dong, for their contributions to the development of Rethlas and its customized version used for the problem studied in this paper.
The authors would like to thank Ruochuan Liu and Gang Tian for constant support and encouragement.

\section{Frame operators and operator averaging}\label{sec:averaging}

Throughout, $X$ is a compact complex manifold, $\mu$ is a smooth positive volume measure on $X$, and $L$ is a basepoint-free holomorphic line bundle on $X$. We let $G$ be a compact group acting holomorphically on $X$ and linearly on $L$, and we fix a smooth $G$-invariant Hermitian metric $h$ on $L$. We write $E=H^0(X,L)$ and $N=\dim E$, equipped with a $G$-invariant Hermitian inner product $\langle\cdot,\cdot\rangle_E$, and we let
\[
\rho:G\longrightarrow U(E)
\]
be the resulting unitary representation on sections. Specializing to $X$ Fano and $L=-kK_X$ recovers the setting of Section~\ref{sec:introduction}; we work in this generality throughout.

For each $x\in X$ the evaluation map $\mathrm{ev}_x:E\to L_x$ defines a positive semidefinite Hermitian operator $Q_x\in\End(E)$ by
\begin{equation}\label{eq:Qx}
\langle Q_x s,t\rangle_E=h_x\bigl(s(x),t(x)\bigr),\qquad s,t\in E.
\end{equation}
For a positive definite Hermitian operator $A\in\End(E)$ we set
\begin{equation}\label{eq:FA}
F_A(x)=\tr(A\,Q_x),\qquad x\in X.
\end{equation}
For a nonzero linear subspace $W\subset E$ with orthonormal basis $w_1,\dots,w_p$ we set
\begin{equation}\label{eq:FW}
F_W(x)=\sum_{j=1}^{p}|w_j(x)|_h^2=\tr(P_W\,Q_x),
\end{equation}
where $P_W$ is the orthogonal projection of $E$ onto $W$; the second equality is the case $A=P_W$ of \eqref{eq:FA} interpreted with $Q_x$ positive semidefinite, and $F_W$ is independent of the chosen orthonormal basis because two such bases differ by a unitary matrix. Since $L$ is basepoint-free, $F_E$ is a positive smooth function on the compact space $X$.

The next statement identifies the invariant Bergman potentials \eqref{eq:Hspace} with the commuting frame operators. With the notation \eqref{eq:FA}, write
\begin{equation}\label{eq:HFA}
\HH_{k,G}=\left\{\tfrac1k\log F_A:\ A\in\End(E)\text{ positive definite Hermitian and }F_A\text{ is }G\text{-invariant}\right\}.
\end{equation}
That \eqref{eq:HFA} agrees with \eqref{eq:Hspace} is the first assertion of the following lemma, applied with $A=B^*B$ for $\{s_i\}$ the rows of $B$ in a fixed orthonormal basis.

\begin{lem}[Operator averaging]\label{lem:averaging}
With the notation above:
\begin{enumerate}
\item[(1)] If $A$ is positive definite Hermitian and $A=B^*B$, then for the basis $s_i=\sum_j B_{ij}e_j$ of $E$ obtained from an orthonormal basis $e_1,\dots,e_N$ of $E$ one has $F_A=\sum_{i=1}^N|s_i|_h^2$; conversely every basis of $E$ arises this way from a unique positive definite Hermitian $A$. Thus \eqref{eq:Hspace} and \eqref{eq:HFA} define the same set $\HH_{k,G}$.
\item[(2)] For every $g\in G$ and $x\in X$ one has the covariance identity $Q_{gx}=\rho(g)\,Q_x\,\rho(g)^{-1}$, and hence $F_A(gx)=F_{\rho(g)^{-1}A\rho(g)}(x)$.
\item[(3)] If $F_A$ is $G$-invariant, then the Haar average
\begin{equation}\label{eq:Abar}
\overline A=\int_G\rho(g)^{-1}A\,\rho(g)\,dg
\end{equation}
with respect to the normalized Haar probability measure $dg$ on $G$ is positive definite Hermitian, commutes with $\rho(G)$, and satisfies $F_{\overline A}=F_A$. Conversely, if $A$ commutes with $\rho(G)$, then $F_A$ is $G$-invariant.
\end{enumerate}
\end{lem}

\begin{proof}
\emph{(1).} Fix an orthonormal basis $e_1,\dots,e_N$ of $E$. In this basis $Q_x$ is the Gram matrix of the values $e_j(x)$, in the sense that $\langle Q_x u,v\rangle_E=h_x\bigl(\sum_j u_je_j(x),\sum_j v_je_j(x)\bigr)$ for all coefficients $u_j,v_j$. If $A=B^*B$ and $s_i=\sum_j B_{ij}e_j$, then
\[
\sum_{i=1}^N|s_i(x)|_h^2=\sum_i h_x\Bigl(\sum_j B_{ij}e_j(x),\sum_l B_{il}e_l(x)\Bigr)=\tr(B^*B\,Q_x)=\tr(A\,Q_x)=F_A(x).
\]
Because $B$ is invertible, its rows $s_i$ are linearly independent and form a basis of $E$. Conversely, any basis $s_i=\sum_j B_{ij}e_j$ has $B$ invertible, and its squared-norm sum is $F_A$ with $A=B^*B$ positive definite Hermitian, uniquely determined by the basis. Comparing \eqref{eq:Hspace} with \eqref{eq:HFA} gives the asserted equality of sets.

\emph{(2).} The $G$-invariance of $h$ and the definition of the action on sections give, for $s,t\in E$,
\[
\langle Q_{gx}s,t\rangle_E=h_{gx}\bigl(s(gx),t(gx)\bigr)=h_x\bigl((\rho(g)^{-1}s)(x),(\rho(g)^{-1}t)(x)\bigr)=\langle Q_x\rho(g)^{-1}s,\rho(g)^{-1}t\rangle_E,
\]
which is the covariance identity $Q_{gx}=\rho(g)Q_x\rho(g)^{-1}$. Taking traces,
\[
F_A(gx)=\tr(A\,Q_{gx})=\tr\bigl(A\,\rho(g)Q_x\rho(g)^{-1}\bigr)=\tr\bigl(\rho(g)^{-1}A\rho(g)\,Q_x\bigr)=F_{\rho(g)^{-1}A\rho(g)}(x).
\]

\emph{(3).} The integrand in \eqref{eq:Abar} is positive definite Hermitian for every $g$, so $\overline A$ is positive definite Hermitian. The invariance of Haar measure gives $\rho(g_0)^{-1}\overline A\,\rho(g_0)=\overline A$ for every $g_0\in G$, i.e.\ $\overline A$ commutes with $\rho(G)$. Using the linearity of $A\mapsto F_A$ and part (2),
\[
F_{\overline A}(x)=\int_G F_{\rho(g)^{-1}A\rho(g)}(x)\,dg=\int_G F_A(gx)\,dg=F_A(x),
\]
the last step using the assumed $G$-invariance of $F_A$. Conversely, if $A$ commutes with $\rho(G)$, then $\rho(g)^{-1}A\rho(g)=A$ and part (2) gives $F_A(gx)=F_A(x)$ for all $g,x$.
\end{proof}

The averaging does not replace a potential by an approximation: when $F_A$ is invariant, $\overline A$ produces the very same function $F_A$. We record one further normalization that we use repeatedly. Multiplying $A$ by a positive constant multiplies $F_A$ by that constant and hence adds a constant to $\tfrac1k\log F_A$; therefore it leaves $\varphi-\sup_X\varphi$ unchanged. In particular, in computing the integrals in \eqref{eq:alpha} we may always normalize a commuting frame operator to have $\tr A=1$.

\section{The upper bound}\label{sec:upper}

\begin{prop}\label{prop:upper}
In the setting of Section~\ref{sec:averaging}, $\alpha_{k,G}\le\glct_{k,G}$.
\end{prop}

\begin{proof}
If $\inf_W\lct|W|=+\infty$, then $\glct_{k,G}=+\infty$ and there is nothing to prove. Assume $\inf_W\lct|W|<+\infty$, and let $c$ be any real number with $c>\glct_{k,G}$; put $a=c/k$. By the definition \eqref{eq:glct} of the infimum there is a nonzero $\rho(G)$-invariant subspace $W\subset E$ with $a>\lct|W|$. Since $L$ is basepoint-free, $F_E$ is positive and smooth on the compact $X$, so $\lct|E|=+\infty$; hence $W\neq E$.

Because $\rho$ is unitary and $W$ is $\rho(G)$-invariant, the orthogonal complement $W^\perp$ is also $\rho(G)$-invariant: if $u\in W^\perp$ and $w\in W$ then $\langle\rho(g)u,w\rangle_E=\langle u,\rho(g)^{-1}w\rangle_E=0$. For $\varepsilon>0$ set
\begin{equation}\label{eq:Aeps}
A_\varepsilon=P_W+\varepsilon P_{W^\perp},
\end{equation}
which is positive definite Hermitian and commutes with $\rho(G)$ because $W$ and $W^\perp$ are $\rho(G)$-invariant. By Lemma~\ref{lem:averaging}(1)(3) the potential $\varphi_\varepsilon=\tfrac1k\log F_{A_\varepsilon}$ lies in $\HH_{k,G}$, and from \eqref{eq:FA} and \eqref{eq:FW},
\begin{equation}\label{eq:FAeps}
F_{A_\varepsilon}=F_W+\varepsilon F_{W^\perp}.
\end{equation}


Since $a>\lct|W|$, the function $F_W^{-a}$ is not locally integrable on $X$, and because $\mu$ is a smooth positive volume measure this gives $\int_X F_W^{-a}\,d\mu=+\infty$. As $\varepsilon\downarrow0$, by \eqref{eq:FAeps} the functions $F_{A_\varepsilon}$ decrease pointwise to $F_W$, so $F_{A_\varepsilon}^{-a}$ increases pointwise to $F_W^{-a}$, and the monotone convergence theorem yields
\begin{equation}\label{eq:mct}
\int_X F_{A_\varepsilon}^{-a}\,d\mu\longrightarrow+\infty\qquad(\varepsilon\downarrow0).
\end{equation}
Let $M_\varepsilon=\max_X F_{A_\varepsilon}$ and $M_0=\max_X F_W$. The functions $F_{A_\varepsilon}$ converge uniformly to $F_W$ on the compact $X$ because $F_{W^\perp}$ is continuous; since $W\neq0$, the function $F_W$ is positive somewhere and $M_0>0$. Hence $M_\varepsilon\to M_0>0$, so the $M_\varepsilon$ are bounded below by a positive constant for all small $\varepsilon$. For these $\varepsilon$,
\[
\int_X e^{-c(\varphi_\varepsilon-\sup_X\varphi_\varepsilon)}\,d\mu
=\int_X\Bigl(\frac{F_{A_\varepsilon}}{M_\varepsilon}\Bigr)^{-c/k}\,d\mu
=M_\varepsilon^{\,a}\int_X F_{A_\varepsilon}^{-a}\,d\mu\longrightarrow+\infty
\]
by \eqref{eq:mct}. Thus $c$ is not an admissible exponent in \eqref{eq:alpha}. Since every $c>\glct_{k,G}$ fails to be admissible, $\alpha_{k,G}\le\glct_{k,G}$.
\end{proof}

\section{Uniform integrability over invariant subspaces}\label{sec:uniform}

The lower bound requires more than the integrability of each individual $F_W^{-a}$: it requires a single bound uniform over the family of invariant subspaces. The mechanism is the effective semicontinuity of complex singularity exponents of Demailly and Koll\'ar.

Recall that for a plurisubharmonic function $\psi$ on a complex manifold $Y$ and a compact $K\subset Y$, the \emph{complex singularity exponent} is
\[
c_K(\psi)=\sup\{b\ge0:\ e^{-2b\psi}\text{ is integrable on a neighborhood of }K\}.
\]
We use the following theorem.

\begin{thm}[Demailly--Koll\'ar {\cite[Main Theorem~0.2(2)]{DK01}}]\label{thm:DK}
Let $Y$ be a complex manifold and $K\subset Y$ compact. Let $\psi$ be a locally $L^1$ plurisubharmonic function on $Y$, and let $b<c_K(\psi)$. If $\psi_m$ is a sequence of locally $L^1$ plurisubharmonic functions converging to $\psi$ in $L^1$ on compact subsets of $Y$, then $e^{-2b\psi_m}\to e^{-2b\psi}$ in $L^1$ on some neighborhood of $K$.
\end{thm}

Let $\Gr(p,E)$ denote the Grassmannian of $p$-dimensional linear subspaces of $E$. We organize the invariant subspaces into a compact parameter space. For each $1\le p\le N$, the set of $\rho(G)$-invariant subspaces $W\in\Gr(p,E)$ is the locus where $\rho(g)P_W\rho(g)^{-1}=P_W$ for all $g\in G$, a closed condition on $P_W$; hence it is a closed, therefore compact, subset of the compact Grassmannian $\Gr(p,E)$. Writing $\lct(W)$ for the exponent \eqref{eq:lctW}, set
\begin{equation}\label{eq:Sset}
S=\bigsqcup_{p=1}^{N}\{W\in\Gr(p,E):\ \rho(g)W=W\ \forall g\in G\},
\end{equation}
a finite disjoint union of compact sets, hence compact.

\begin{prop}[Uniform integrability]\label{prop:uniform}
In the setting of Section~\ref{sec:averaging}, let $S$ be as in \eqref{eq:Sset} and put $\ell_G=\inf_{W\in S}\lct(W)$, so that $\glct_{k,G}=k\,\ell_G$. If $a$ is a real number with $0<a<\ell_G$, then
\begin{equation}\label{eq:uniform}
\sup_{W\in S}\int_X F_W^{-a}\,d\mu<\infty.
\end{equation}
\end{prop}

\begin{proof}
Fix $W_0\in S$ and put $p=\dim W_0$; we work inside the component $\Gr(p,E)$ containing $W_0$. For $W$ near $W_0$ the orthogonal projection $W_0\to W$ is an isomorphism, so applying Gram--Schmidt to the images of a fixed orthonormal basis $u_1,\dots,u_p$ of $W_0$ produces an orthonormal basis $w_1(W),\dots,w_p(W)$ of $W$ that depends continuously on $W$, with $w_j(W)\to u_j$ in $E$ as $W\to W_0$.

Let $V\subset X$ be a coordinate open set on which $L$ has a holomorphic frame $\tau$, and write $w_j(W)=f_j(W,z)\,\tau$ on $V$. Since $E$ is finite-dimensional and $W\mapsto w_j(W)$ is continuous into $E$, the holomorphic coefficients $f_j(W,\cdot)$ converge uniformly on compact subsets of $V$ to $f_j(W_0,\cdot)$ as $W\to W_0$. Define the local plurisubharmonic function
\begin{equation}\label{eq:psiW}
\psi_W=\frac12\log\sum_{j=1}^p|f_j(W,z)|^2.
\end{equation}
The sections $w_j(W)$ are not all zero, so $\psi_W\not\equiv-\infty$, and the uniform convergence of the coefficients gives $\psi_W\to\psi_{W_0}$ in $L^1_{\mathrm{loc}}(V)$ as $W\to W_0$.

In the frame $\tau$ we have $F_W=|\tau|_h^2\sum_j|f_j(W,z)|^2$, so $F_{W_0}^{-a}=|\tau|_h^{-2a}\,e^{-2a\psi_{W_0}}$ with $|\tau|_h^{-2a}$ smooth, positive, and (together with the density of $\mu$ relative to Lebesgue measure) bounded above and below on any compact $K\subset V$. Hence, for $a<\lct(W_0)$, local integrability of $F_{W_0}^{-a}$ on $X$ from \eqref{eq:lctW} gives integrability of $e^{-2a\psi_{W_0}}$ on a neighborhood of $K$, i.e.\ $a\le c_K(\psi_{W_0})$; in particular $a<\lct(W_0)\le c_K(\psi_{W_0})$, so Theorem~\ref{thm:DK} applies with $b=a$ and $\psi=\psi_{W_0}$. We claim it yields a Grassmannian neighborhood of $W_0$ on which $\int_K F_W^{-a}\,d\mu$ is uniformly bounded. If not, there would be $W_m\to W_0$ in $\Gr(p,E)$ with $\int_K F_{W_m}^{-a}\,d\mu$ unbounded; but then $\psi_{W_m}\to\psi_{W_0}$ in $L^1_{\mathrm{loc}}$, so Theorem~\ref{thm:DK} forces $e^{-2a\psi_{W_m}}\to e^{-2a\psi_{W_0}}$ in $L^1$ on a neighborhood of $K$, hence $\int_K F_{W_m}^{-a}\,d\mu$ converges, a contradiction.

Cover $X$ by finitely many coordinate sets $V_1,\dots,V_r$ with compact $K_i\subset V_i$ whose interiors cover $X$. Applying the previous paragraph to each $K_i$ and intersecting the finitely many Grassmannian neighborhoods produces a neighborhood $\mathcal U_{W_0}$ of $W_0$ and a constant $C_{W_0}$ with $\int_X F_W^{-a}\,d\mu\le C_{W_0}$ for all $W\in\mathcal U_{W_0}$. As $W_0$ ranges over the compact set $S$, the neighborhoods $\mathcal U_{W_0}$ cover $S$; a finite subcover yields finitely many constants whose maximum bounds $\int_X F_W^{-a}\,d\mu$ for every $W\in S$. This is \eqref{eq:uniform}.
\end{proof}

\section{The lower bound}\label{sec:lower}

\begin{prop}\label{prop:lower}
In the setting of Section~\ref{sec:averaging}, $\alpha_{k,G}\ge\glct_{k,G}$.
\end{prop}

\begin{proof}
Let $S$ and $\ell_G$ be as in Proposition~\ref{prop:uniform}, so $\glct_{k,G}=k\,\ell_G$. If $\ell_G=0$ the inequality is immediate, since $\alpha_{k,G}\ge0$. Assume $\ell_G>0$, let $c$ be real with $0<c<k\,\ell_G$, and put $a=c/k$, so $0<a<\ell_G$. By Proposition~\ref{prop:uniform} there is a finite constant $C_a$ with
\begin{equation}\label{eq:Ca}
\int_X F_W^{-a}\,d\mu\le C_a\qquad\text{for every }W\in S.
\end{equation}

Let $\varphi\in\HH_{k,G}$. By Lemma~\ref{lem:averaging}, $\varphi=\tfrac1k\log F_B$ for a positive definite Hermitian $B$ with $F_B$ $G$-invariant, and the Haar average produces a positive definite Hermitian $A_0$ commuting with $\rho(G)$ and satisfying $F_{A_0}=F_B$. Replacing $A_0$ by $A=A_0/\tr(A_0)$ multiplies $F_{A_0}$ by a positive constant and so leaves $\varphi-\sup_X\varphi$ unchanged. Thus we may assume
\begin{equation}\label{eq:Anorm}
\varphi=\tfrac1k\log F_A,\qquad A\text{ positive definite Hermitian},\qquad A\text{ commutes with }\rho(G),\qquad\tr A=1.
\end{equation}

Let $\lambda_{\max}$ be the largest eigenvalue of $A$ and $W_{\mathrm{top}}=\{v\in E:Av=\lambda_{\max}v\}$ its eigenspace. Since $A$ commutes with $\rho(G)$, for every $g\in G$ and $v\in W_{\mathrm{top}}$ one has $A(\rho(g)v)=\rho(g)Av=\lambda_{\max}\rho(g)v$, so $\rho(g)W_{\mathrm{top}}\subset W_{\mathrm{top}}$; applying this to $g^{-1}$ gives equality, whence $W_{\mathrm{top}}$ is $\rho(G)$-invariant and $W_{\mathrm{top}}\in S$. Because $A$ is positive definite with $N$ positive eigenvalues summing to $1$, we have $\lambda_{\max}\ge1/N$, and therefore, in the Hermitian operator order,
\begin{equation}\label{eq:Adom}
A\ge\frac1N\,P_{W_{\mathrm{top}}}.
\end{equation}
Since $Q_x$ is positive semidefinite, pairing \eqref{eq:Adom} against $Q_x$ in \eqref{eq:FA} gives
\begin{equation}\label{eq:Flower}
F_A(x)=\tr(A\,Q_x)\ge\frac1N\tr(P_{W_{\mathrm{top}}}Q_x)=\frac1N\,F_{W_{\mathrm{top}}}(x),\qquad x\in X.
\end{equation}
On the other hand every eigenvalue of $A$ is at most $1$, so $A\le I_E$, whence $F_A(x)\le F_{I_E}(x)=F_E(x)$ for all $x$. Put $C_0=\max_X F_E$, which is finite because $X$ is compact and $F_E$ is continuous, and let $M_A=\max_X F_A$, so $M_A\le C_0$. Combining with \eqref{eq:Flower},
\begin{equation}\label{eq:pointwise}
e^{-c(\varphi-\sup_X\varphi)}=\Bigl(\frac{F_A}{M_A}\Bigr)^{-a}=M_A^{\,a}\,F_A^{-a}\le C_0^{\,a}N^{a}\,F_{W_{\mathrm{top}}}^{-a}.
\end{equation}
Integrating \eqref{eq:pointwise} and using \eqref{eq:Ca} with $W=W_{\mathrm{top}}\in S$,
\[
\int_X e^{-c(\varphi-\sup_X\varphi)}\,d\mu\le(NC_0)^{a}\int_X F_{W_{\mathrm{top}}}^{-a}\,d\mu\le(NC_0)^{a}\,C_a.
\]
The right-hand side depends on $c$ and the fixed data but not on $\varphi$. Hence
\[
\sup_{\varphi\in\HH_{k,G}}\int_X e^{-c(\varphi-\sup_X\varphi)}\,d\mu<\infty,
\]
so $c$ is admissible in \eqref{eq:alpha}. Every $c$ with $0<c<k\,\ell_G$ is admissible, so $\alpha_{k,G}\ge k\,\ell_G=\glct_{k,G}$.
\end{proof}

\section{Proof of the main theorem}\label{sec:proof}

We first prove the general form of the result, of which Theorem~\ref{thm:main} is the special case $X$ Fano and $L=-kK_X$.



\begin{proof}[Proof of Theorem~\ref{thm:general}]
Proposition~\ref{prop:upper} gives $\alpha_{k,G}\le\glct_{k,G}$, and Proposition~\ref{prop:lower} gives $\alpha_{k,G}\ge\glct_{k,G}$. As these are inequalities between the same two elements of the extended nonnegative reals, $\alpha_{k,G}=\glct_{k,G}$.
\end{proof}

\begin{proof}[Proof of Theorem~\ref{thm:main}]
Let $X$ be a Fano manifold, $G\subset\Aut(X)$ a compact subgroup, and $k$ a positive integer with $-kK_X$ basepoint-free. Then $X$ is compact, and since $G$ is compact we may choose a smooth $G$-invariant Hermitian metric $h$ on $L=-kK_X$, a $G$-invariant Hermitian inner product on $E=H^0(X,-kK_X)$ inducing the unitary representation $\rho$, and a smooth positive volume measure $\mu$. With this data the hypotheses of Theorem~\ref{thm:general} hold, and the definitions \eqref{eq:alpha} and \eqref{eq:glct} are those of \cite[Definition~4.1, Lemma~4.3, and Definition~4.4]{JR25}. Theorem~\ref{thm:general} therefore gives $\alpha_{k,G}=\glct_{k,G}$, which is \eqref{eq:equality}. As $L=-kK_X$ is the $k$-th tensor power of $-K_X$, this is the fixed-level anticanonical equality of Question~\ref{ques:jr}, answered affirmatively for every Fano manifold $X$ and every compact subgroup $G\subset\Aut(X)$.
\end{proof}

\begin{rem}\label{rem:scope}
The identity is genuinely a fixed-level statement: the proof compares $\alpha_{k,G}$ and $\glct_{k,G}$ at one and the same $k$ and does not pass to the limit $k\to\infty$. The only external analytic input is the Demailly--Koll\'ar effective semicontinuity theorem \cite[Main Theorem~0.2(2)]{DK01}, used in Proposition~\ref{prop:uniform}; the remaining ingredients are the operator-averaging parametrization of invariant Bergman potentials and the top-eigenspace domination \eqref{eq:Adom}. None of these uses a toric structure, so the toric hypothesis in \cite[Proposition~1.3]{JR25} is removed.
\end{rem}

\end{document}